\documentclass[11pt]{article}
\usepackage{amsfonts,amssymb,amscd,amsmath,enumerate,verbatim,calc, times,url}
\usepackage{amsthm, psfrag,latexsym,epsfig,mdwlist,graphicx}

\theoremstyle{plain}
\newtheorem{theorem}{Theorem}[section]
\newtheorem{lemma}[theorem]{Lemma}
\newtheorem{proposition}[theorem]{Proposition}

\theoremstyle{definition}
\newtheorem{definition}[theorem]{Definition}

\newtheorem{remark}[theorem]{Remark}

\newtheorem{question}[theorem]{Question}

\newcommand{\xs}{x_1,\ldots,x_n}                
\newcommand{\Fs}{F_1,\ldots,F_q}                

\newcommand{\m}{\mathbf{m}}                 
\newcommand{\dimn}{\operatorname{dim}}
\newcommand{\rmv}[1]{\setminus \langle #1\rangle}
\newcommand{\F}{{\mathcal{F}}}                  
       
\newcommand{\D}{\Delta}                         
\newcommand{\lcm}{{\mathop{\rm{lcm}}}}          
              
\newcommand{\st}{\ | \ }                        
\newcommand{\tuple}[1]{\langle #1 \rangle}      
\newcommand{\void}[1]{}

\newcommand{\erase}[1]{}

\newcommand{\LCM}{\mbox{LCM}} 
\newcommand{\facets}{\mbox{Facets}} 
\newcommand{\NN}{\mathbb{N}} 
\newcommand{\fbar}{\overline{F}}

\newcommand{\cocoa}{\mbox{\rm C\kern-.13em o\kern-.07 em C\kern-.13em o\kern-.15em A}} 
\newcommand{\cocoax}{\mbox{C\kern-.13em o\kern-.07 em C\kern-.13em o\kern-.15em A}} 
\newcommand{\cocoal}{\mbox{\rm C\kern-.13em o\kern-.07 em C\kern-.13em o\kern-.15emA\kern-.1em L}}

\newcommand{\idiot}[1]{\vspace{5 mm}\par \noindent
\framebox{\begin{minipage}[c]{0.95 \textwidth}
\tt #1 \end{minipage}}\vspace{5 mm}\par}

\renewcommand{\idiot}[1]{}

\newcommand{\sm}{\setminus}

\renewcommand{\leq}{\leqslant}          
\renewcommand{\geq}{\geqslant}  

\date{}

\author{Sara Faridi\thanks{Department of
Mathematics and Statistics, Dalhousie University, Halifax, Canada, 
faridi@mathstat.dal.ca. 
Research supported by NSERC.}}

\title{\Large \sc Lattice Complements and the
  Subadditivity of Syzygies of Simplicial Forests}
\begin{document}

\maketitle
\begin{abstract}

We prove the subadditivity property for the maximal degrees of the
syzygies of facet ideals simplicial forests.  For such an ideal $I$,
if the $i$-th Betti number is nonzero and $i=a+b$, we show that there
are monomials in the lcm lattice of $I$ that are complements in part
of the lattice, each supporting a nonvanishing $a$-th and 
$b$-th Betti numbers. The subadditivity formula follows from this
observation.

 \end{abstract}

\section{Introduction} 

Let $=k[\xs]$, where $k$ is a field, and let $I$ be a graded ideal of
$S$, and suppose $S/I$  has  minimal graded free resolution
$$ 0 \to \oplus_{j \in \NN} S(-j)^{\beta_{p,j}}\to \oplus_{j \in \NN}
S(-j)^{\beta_{p-1,j}} \to \cdots \to \oplus_{j \in \NN}
S(-j)^{\beta_{1,j}} \to S$$ where the graded Betti numbers
$\beta_{i,j}(S/I)$ denote the rank of the degree $j$ component $S(-j)$
appearing in the $i$th homological degree in this sequence.

For an integer $i$, define
$$t_a(I)= \max \{ j \st \beta_{a,j}(S/I) \neq 0 \}. $$ 

We say that the degrees of the Betti numbers of $I$ satisfy the
\emph{subadditivity property } if $$t_{a+b}(I) \leq t_a(I) + t_b(I)$$
for all $a,b >0$ with $a+b \leq p$, where $p$ is the projective
dimension of $S/I$.

It is known that in general the subadditivity property does not
hold~(Avramov, Conca, Iyengar~\cite{ACI}), but under restrictive
conditions, many cases have been known to hold. These include some
algebras of krull dimension at most~1~(Eisenbud, Huneke and
Ulrich~\cite{EHU}), when $a=1$ (Fern\'andez-Ramos and Gimenez
~\cite{FG} if $I$ is generated by degree~2 monomials, Herzog and
Srinivasan~\cite{HS} when $b$ is the projective dimension of $S/I$ or
when $I$ is any monomial ideal), in certain homological degrees in the
case of Gorenstein algebras (El Khoury and Srinivasan~\cite{ES}) and
the case $a=1,2,3$ for a monomial ideal generated in degree
2~(Abedelfatah and Nevo~\cite{AN}).  Otherwise, the question is wide
open for the class of monomial ideals.

We approach this problem using ``lattice complements'' which appear in the
topology of lattices. Two elements of a lattice are complements if
their join and their meet are $\hat{1}$ and $\hat{0}$,
respectively. In the case of two monomials in the lcm lattice of a
monomial ideal, they are complements if their $\gcd$ is not in the
ideal and their lcm is the lcm of all the generators.

Our motivation for using complements is the fact that if the top
degree Betti number of a monomial ideal $I$ is nonzero, then every
monomial in the lcm lattice of $I$ has a complement. This follows from
the interpretation of Betti numbers of monomial ideals in terms of
homology of open intervals in lattices by Gasharov, Peeva and
Welker~(\cite{GPW,P}) and Baclawski's~(\cite{B}) work that relates the
homology of lattices to the existence of complements.

By polarization, to find Betti numbers of monomial ideals it is enough
to consider Betti numbers of square-free monomial ideals. Moreover, in
this case inquiries about a specific graded Betti number reduces to
that of ``top degree'' Betti numbers -- see below for more on this.

So we ask the following question.

\begin{question}\label{q:main} If $I$ is a square-free monomial involving $n$
  variables and $\beta_{i,n} (S/I)\neq 0$, $a,b >0$ and $i=a+b$, are
  there complements $\m$ and $\m'$ in the lcm lattice of $I$ with
  nonzero multigraded Betti numbers $\beta_{a,\m}(S/I)$ and
  $\beta_{b,\m'}(S/I)$?
\end{question}

If the answer is positive, then the subadditivity conjecture is true
for all monomial ideals, since $\deg(\m)+ \deg(\m') \geq n$.

In this paper we give a positive answer to this question in the case
where $I$ is the facet ideal of a simplicial forest $\D$. In this
case, we take advantage of the a recursive formula relating Betti
numbers in each homological degree to lower ones~\cite{HV,F2}. What we
really use is the inductive existence of a facet with a free vertex,
which acts as a splitting facet (in the Eliahou-Kervaire sense). We
hope, however, that some of these methods can be used to study the
general version of the subadditivity property for monomial ideals. A
further study to see if Question~\ref{q:main} holds in general would be
extremely useful for this purpose.

\noindent {\bf Acknowledgements:} The author thanks Volkmar Welker for
introducing her to lattice complements, the math department at TU
Darmstadt for their hospitality while this research was done, and the
Canadian funding agency NSERC for their financial support.

\section{Setup}

A {\bf simplicial complex} $\D$ is a set of subsets of a set $A$, such
that if $F \in \D$ then all subsets of $F$ are also in $\D$. Every
element of $\D$ is called a {\bf face} of $\D$, the maximal elements
under inclusion are called {\bf facets} and the {\bf dimension of a
face $F$} of $\D$ is defined as $|F| -1$.  The faces of dimensions 0
and 1 are called {\bf vertices} and {\bf edges}, respectively, and
$\dimn \emptyset =-1$.  The {\bf dimension of $\D$} is the maximal
dimension of its facets.  We denote the set of vertices of $\D$ by
$V(\D)$. 

A {\bf subcollection} of $\D$ is a simplicial complex whose facets are
also facets of $\D$; in other words a simplicial complex generated by
a subset of the set of facets of $\D$. If $u \subseteq V(\D)$, then
the subcollection $\D_{[u]}$ consisting of all facets of $\D$
contained in $u$ is an {\bf induced subsollection of $\D$}.

We denote the set of facets of $\D$ by
$\facets(\D)$. If $\facets(\D)=\{\Fs\}$, we write
$\D=\tuple{\Fs}$. The
  simplicial complex obtained by {\bf removing the facet} $F_i$ from
  $\D$ is 
 $$\D \rmv{F_i}=\tuple{F_1,\ldots,\hat{F}_{i},\ldots,F_q}.$$ 

  If $F$ is a facet of $\D$, then $\fbar=V(\D) \sm F$, i.e. all
  vertices of $\D$ that are not in $F$.

  A facet $F$ of $\D$ simplicial complex is called a {\bf leaf}
  if either $F$ is the only facet of $\D$ or for some facet
  $G \in \D \rmv{F}$ we have $F \cap (\D \rmv{F}) \subseteq G.$
  Equivalently, we can say a facet $F$ is a leaf of $\D$ if
  $F \cap (\D \rmv{F})$ is a face of $\D \rmv{F}$.  It follows
  immediately from the definition above that a leaf $F$ must contain
  at least one vertex that belongs to no other facet of $\D$ but $F$;
  we call such a vertex a {\bf free vertex}.

  We call $\D$ a {\bf simplicial forest} if every nonempty
  subcollection of $\D$ has a leaf.  A connected simplicial forest is
  called a {\bf simplicial tree}.

Let $S=k[\xs]$ and $I$ a square-free monomial ideal in $S$.  For a
subset $u \subset \{\xs\}$, we denote by $\m_u$ the  {\bf  square-free
monomial with  support $u$}, that is $$\m_u= \Pi_{x_i \in u} x_i.$$

The {\bf facet complex of $I$}, denoted is the simplicial
complex $$\F(I)=\tuple{u \st \m_u \mbox{ is a generator of } I}.$$
Conversely, given a simplicial complex $\D$ on vertices from the set
$\{\xs\}$, we can define the {\bf facet ideal of $\D$} as
$$\F(\D)=(\m_u \st u \mbox{ is a facet of } \D),$$ which is an ideal of $S$.

One of the properties of simplicial trees that we will be using in
this article is the following.

\begin{lemma}[Localization of a forest is a
  forest~\cite{F1}]\label{l:localization} Suppose $\Delta$ is a
  simplicial forest with facet ideal $I$ in the polynomial ring
  $S$. Then for any prime ideal $p$ of $S$, $I_p$ is the facet ideal
  of a simplicial forest which we denote by $\D_p$. \end{lemma}

For a simplicial complex $\D$ with $I=\F(\D) \subseteq S$, by
$\beta_{i,j}(\D)$ we mean $\beta_{i,j}(S/I)$. The localization
property has a substantial effect on the calculation of Betti numbers
of forests.

If $I$ is a square free monomial ideal in the polynomial ring
$S=k[\xs]$ and with facet complex $\D$, then every graded Betti number
$\beta_{i,j}(S/I)$ is calculated by taking the sum of all
\emph{multigraded} Betti numbers $\beta_{i,\m}(S/I)$ where $\m$ is a
square-free monomial in $S$ of degree $j$ (see for
example~\cite{P}). Such a monomial $\m$ is in fact $\m_u$ for some
$u\subseteq \{\xs\}$, and using for example the Taylor complex~\cite{T,P}, one
can see that
$$\beta_{i,\m_u}(\D)= \beta_{i,j}(\D_{[u]})=\beta_{i,|u|}(\D_{[u]})$$
where $\D_{[u]}$ is the induced subcollection of $\D$ on $u$, and
$j=|u|$. The Taylor resolution also shows that if $\beta_{i,\m_u}(\D)
\neq 0$, then $\D_{[u]}$ must have exactly $u$ vertices. These observations
reduce the calculation of $\beta_{i,j}(\D)$ to the calculation of the
``top degree'' Betti numbers of certain subcollections~(see Remark~2.3
and Lemma~3.1 of~\cite{EF1}).

When $I$ is the facet ideal of a simplicial tree $\D$, much more about
$\beta_{i,j}(\D)$ is known, see for example~\cite{EF1,EF2}. In
particular, a recursive formula for the calculation of the Betti
numbers of trees in~\cite{F2}, which is deduced from a splitting
formula due to H\`a and Van Tuyl (Theorem~5.5 of~\cite{HV}) can be
used effectively in this case.

If $\D$ is a connected simplicial complex with facet ideal $I$ and
$F$ is facet of $\D$ with a free vertex (for example a leaf of a tree), then
H\`a and Van Tuyl's Theorem~5.5 gives (\cite{F2}), for $i,j\geq 0$
\begin{align}\label{e:glo}\beta_{i,j}(\D)=
\beta_{i,j}(\D\rmv{F})+ \beta_{i-1,j-|F|}((\D\rmv{F})_{\fbar}).
\end{align}
\idiot{Ha and Van Tuyl have this only for $i \geq 1$, but one can show
  that it also works for $i=0,1$ by a direct argument.}
  
In what follows, we will use the localized complex
$\Gamma=(\D\rmv{F})_{\fbar}$ extensively. It is worth observing that
if $\D=\tuple{\Fs}$ then $\Gamma$ has as facets the minimal elements,
under inclusion, of $$F_1 \sm F, \ldots, F_q \sm F$$ which correspond to
the generators of the ideal $I_{\left (\fbar \right )}$.  This leads to the
following observation, which was used in the case of trees
in~\cite{EF2}.

\begin{proposition}\label{p:local-prop}
  Let $\D$ be a simplicial complex on $n$ vertices, $F$ a facet of $\D$
  with a free vertex, and suppose $\Gamma=\left (\D \rmv{F}\right
  )_{\fbar}$. Then for every $i$ we
  have $$\beta_{i,n}(\D)=\beta_{i-1,n-|F|}(\Gamma).$$
\end{proposition}

    \begin{proof} If $\D$ is connected, then since
      $F$ has a free vertex, $\D\rmv{F}$ has strictly less than $n$
      vertices, so in Equation~(\ref{e:glo})
      $\beta_{i,n}(\D\rmv{F})=0$, which results in
      $\beta_{i,n}(\D)=\beta_{i-1,n-|F|}(\Gamma).$

     Suppose $\D$ has connected components
     $\D_1,\ldots,\D_r$ each with $n_1,\dots,n_r$
     vertices, respectively. Assume, without loss of generality, $F$
     is a facet of $\D_1$.  The connected components of $\Gamma$ are of the form
     $$\Gamma_a= \left ( \D_a\rmv{F}\right ) _{\fbar}$$ where if
     $a>1$, one can see immediately that $\Gamma_a=\D_a$, as $F$ and
     $\D_a$ will have no vertices in common.
    So we can write (see Lemma~3.2
     of~\cite{EF1})
      \begin{align*}\beta_{i,n}(\D)& = \displaystyle \sum_{u_1+\cdots +
          u_r=i}\beta_{u_1,n_1}(\D_1) \ldots \beta_{u_r,n_r}(\D_r) \notag \\
        &=  \displaystyle \sum_{u_1+\cdots +
          u_r=i}\beta_{u_1-1,n_1-|F|}\left ((\D_1 \rmv{F})_{\fbar}\right ) \beta_{u_2,n_2}(\D_2)\ldots \beta_{u_r,n_r}(\D_r) \notag\\
                &=  \displaystyle \sum_{u_1+\cdots +
          u_r=i}\beta_{u_1-1,n_1-|F|}(\Gamma_1) \beta_{u_2,n_2}(\Gamma_2)\ldots \beta_{u_r,n_r}(\Gamma_r) \notag\\
        &= \beta_{i-1,n-|F|}(\Gamma).
      \end{align*}
\idiot{for the last equation, if $\Gamma_1$ has strictly less than
  $n_1-|F|$ vertices, then the Betti numbers will just be 0, and
  equality will still hold.}
   \end{proof}

 \erase{   
    \begin{theorem}[\cite{EF1,EF2}]\label{t:jpaa} Let $I$ be the facet
  ideal of a simplicial forest.
  \begin{enumerate}
    \item $\beta_{a,j}(S/I) \neq 0 $ if and only if $\D$ has an induced
      subcollection $\D_{[W]}$ where \begin{enumerate}
      \item $|W|=j$, and
      \item $\D_{[W]}$ has a
      well-ordered facet cover of cardinality $a$.
      \end{enumerate}
    \item If $\beta_{a,W}(S/I) \neq 0$ for some $a$ and $W$, then
      $\beta_{\ell,W}(S/I) = 0$ if $\ell \neq a$.
    \item $\beta_{a,W}(S/I)$ is always $0$ or $1$.
  \end{enumerate}
\end{theorem}
}

\section{Betti numbers of complements}

Let $I$ be a square-free monomial ideal in $S=k[\xs]$. We denote the
lcm lattice of $I$ by $\LCM(I)$. The atoms of this lattice are the
generators of $I$ and the other members are lcm's of the generators of
$I$ ordered by divisibility. The top element $\hat{1}$ is the lcm of
all generators of $I$, and the bottom element $\hat{0}$ is
$1$. See~\cite{P} for more on lcm lattices and their properties.

The following definition is an adaptation of the usual concept of lattice
complements to the lcm lattice.

\begin{definition} Let $I$ be a square-free monomial ideal. Two
  monomials $\m$ and $\m'$ in $\LCM(I)\sm~\{\hat{0},\hat{1}\}$
  are called {\bf complements} if
\begin{itemize}
\item $\lcm(\m,\m')=\hat{1}$ and 
\item $\gcd(\m,\m') \notin I$.
\end{itemize}
\end{definition}

This definition leads directly to the following statement.

\begin{lemma}\label{l:complement} Let $\D$ be a simplicial complex on 
  $n$ vertices and with facet ideal $I \subseteq S$, and let $u$ and
  $v$ be two proper subsets of $\{\xs\}$, with $\m_u, \m_v \in
  \LCM(I)\sm~\{\hat{0},\hat{1}\}$ . Then the following are equivalent.
\begin{enumerate}
\item $\m_u$ and $\m_v$ are complements in $\LCM(I)$ 
\item \begin{enumerate}
\item $u \cup v = \{\xs\}$, and 
\item the two induced subcollections $\D_{[u]}$ and $\D_{[v]}$ have no facet in common.
\end{enumerate}
\end{enumerate}
\end{lemma}

Based on the observation above, we call two proper induced
subcollections $\D_{[u]}$ and $\D_{[v]}$ of $\D$ {\bf complements in
  $\D$} if the two conditions in Lemma~\ref{l:complement}~(2) hold.

\begin{lemma}\label{l:induced-localization} Let $\D$ be a simplicial 
  complex, $F$ a facet of $\D$ with a free vertex,
  $\Gamma=\left(\D \sm \tuple{F}\right )_{\overline{F}}$, and
  $u \subset \overline{F}$. 
  \begin{enumerate}
   \item $\Gamma_{[u]}= \left(\Delta_{[F\cup u]}\sm \tuple{F}\right ) _
  {\overline{F}}.$
\item If $\Gamma_{[u]}$ has $|u|$ vertices then $\D_{[F\cup u]}$ has
  $|u \cup F|=|u| + |F|$ vertices.
\end{enumerate}
\end{lemma}

     \begin{proof} 

      \begin{enumerate}
       \item We show each inclusion.
      \begin{itemize}

       \item[($\subseteq$)] Let $G \sm F$ be a facet of
         $\Gamma_{[u]}$, then $G$ is a facet of $\D$ such that $(G \sm
         F) \subseteq u$, which implies that $G \subseteq u \cup
         F$. Hence $G$ is a facet of $\D_{[F \cup u]}$. If $H$ is
         another facet of $\D_{[F \cup u]}$ with $(H \sm F) \subsetneq
         (G \sm F)$, then as $(H \sm F) \subseteq u$ we will have $(H
         \sm F) \in \Gamma_{[u]}$ which contradicts the fact
         that $G\sm F$ is a facet of $\Gamma_{[u]}$.

       \item[($\supseteq$)] Suppose $G \sm F$ is a facet of
         $\left(\Delta_{[F\cup u]}\sm \tuple{F}\right ) _
         {\overline{F}}$.
         Then $G \subseteq F \cup u$ and hence
         $(G \sm F) \subseteq u$. If $(G \sm F) \notin \Gamma_{[u]}$
         then there is another facet $H$ of $\D$ with
         $(H \sm F) \subsetneq (G \sm F) \subseteq u$. But then then
         $H\sm F \in \left(\Delta_{[F\cup u]}\sm \tuple{F}\right
         )_{\overline{F}}$, which is a contradiction.
     \end{itemize}

   \item Suppose $\Gamma_{[u]}=\tuple{G_1 \sm F, \ldots, G_t \sm F}$,
     where by the previous part we can pick $G_1,\ldots,G_t$ to be
     facets of $\D_{[F \cup u]}$. Then
     $|(G_1\cup \ldots \cup G_t)\sm F|=|u|$ and also
     $F \in \D_{[F \cup u]}$. Since $u \cap F=\emptyset$ we have
     already accounted for $|F \cup u|$ vertices in $\D_{[F \cup u]}$.
     However this is the maximum number of vertices that
     $\D_{[F \cup u]}$ could have, so our claim is proved.
     \end{enumerate}
     \end{proof}

\begin{theorem}\label{t:comp-facet} Let $\D$ be a simplicial forest
  with $n$ vertices and more than one facet, and suppose
  $\beta_{i,n}(\D)\neq 0$. Then for every facet $G$ of $\D$,
  $\D_{[G]}$ has a complement $\D_{[u]}$ in $\D$ with
  $\beta_{i-1,|u|}(\D_{[u]})\neq~0.$
\end{theorem}

   \begin{proof} We use induction on $n$. The smallest case
      is $n=2$ and $\D$ a complex consisting of two isolated vertices
      $F=\{x_1\}$ and $G=\{x_2\}$, and we have $\beta_{2,2}(\D)\neq 0$
      and $\beta_{1,1}(\D_{[F]}) = \beta_{1,1}(\D_{[G]})\neq 0$. Since
      $\D_{[F]}$ and $\D_{[G]}$ are complements, this settles the
      statement in the case $n=2$. 
  
      Since every forest has at least two leaves~(\cite{F1}), we can
      choose $F$ to be a leaf of $\Delta$ with $F\neq G$.  Suppose
      $\Gamma= \left(\D \sm \tuple{F}\right )_{\overline{F}}$.  By
      Proposition~\ref{p:local-prop}, $\beta_{i,n}(\D)\neq 0$ results
      in $\beta_{i-1,n-|F|}(\Gamma)\neq 0$, which in particular
      implies that $\Gamma$ has $n-|F|$ vertices. 
     
    By Lemma~\ref{l:localization}, $\Gamma$ is a forest, so it
    satisfies the induction hypothesis.

    Now $\Gamma$ has a facet $H\sm F$ such that
    $(H \sm F) \subseteq (G \sm F)$ (with $H=G$ possible). By the
    induction hypothesis, $\Gamma_{[H\sm F]}$ has a complement
    $\Gamma_{[v]}$ in $\Gamma$ such that
    $\beta_{i-2,|v|}(\Gamma_{[v]})\neq 0$, which in particular implies
    that $\Gamma_{[v]}$ has $|v|$ vertices.
     
    By Lemma~\ref{l:induced-localization}, $\D_{[F\cup v]}$ has
    $|F|+|v|$ vertices and
     $$\Gamma_{[v]}=\left(\Delta_{[F\cup v]}\sm \tuple{F}\right ) _
     {\overline{F}}.$$ Proposition~\ref{p:local-prop} now implies that
     $\beta_{i-1,|F \cup v|}(\Delta_{[F\cup v]}) \neq 0$.

       We set $u=F \cup v$ and show that $\D_{[u]}$ and $\D_{[G]}$ are
       complements in $\D$. Since $\Gamma_{[v]}$ and $\Gamma_{[H \sm
           F]}$ are complements in $\Gamma$,
       $$v \cup (H \sm F) = V(\Gamma)=(V(\D) \sm F)$$ and add to this
       the fact that $(H \sm F) \subseteq (G \sm F)$ to conclude $$ u \cup
       G= v \cup F \cup G = V(\D).$$

       If $\D_{[u]}$ and $\D_{[G]}$ have a facet in common, then
       $G \in \D_{[u]}$, which implies that $G \subseteq u=v \cup
       F$. On the other hand
       $$(H \sm F) \subseteq (G \sm F) \subseteq v \Longrightarrow
       (H\sm F) \in \Gamma_{[v]},$$ which contradicts $\Gamma_{[v]}$
       and $\Gamma_{[H \sm F]}$ being complements in $\Gamma$.

       So $\D_{[u]}$ and $\D_{[G]}$ are complements in $\D$ and we are
       done.
    \end{proof}    

   \begin{theorem}\label{t:complement-nonzero} Let $\D$ be a simplicial
     forest with more than one facet and facet ideal $I \subseteq S$. Suppose
     $\beta_{i,n}(\D)\neq 0$, and $i=a+b$ for some positive integers
     $a$ and $b$. Then there are complements $\D_{[u]}$ and $\D_{[w]}$
     in $\D$ with $\beta_{a,|u|}(\D_{[u]}) \neq 0$ and
     $\beta_{b,|w|}(\D_{[w]}) \neq 0$.
\end{theorem}

        \begin{proof} Without loss of generality assume 
          $\D$ has $n$ vertices. 
          
          We prove the statement by induction on $n$. The base
          case is $n=2$, where $\D$ is two isolated vertices. In this
          case $i=2=1+1$, and the claim follows from
          Theorem~\ref{t:comp-facet}.

          We consider  the general case. From Theorem~\ref{t:comp-facet} the
          statement is true if $a=1$ or $b=1$. So we may assume $a,b
          >1$.  

          Let $F$ be a leaf of $\D$ and let
          $\Gamma= \left(\D \sm \tuple{F}\right )_{\overline{F}}$. By
          Proposition~\ref{p:local-prop}
          $\beta_{i-1,n-|F|}(\Gamma)\neq 0$, and in particular
          $|V(\Gamma)|=n-|F|$.  By the induction hypothesis, since
          $i-1=(a-1)+b$, there are complements $\Gamma_{[u']}$ and
          $\Gamma_{[v']}$ in $\Gamma$ such that
          $$\beta_{a-1,|u'|}(\Gamma_{[u']}) \neq 0 \mbox{ and }
          \beta_{b,|v'|}(\Gamma_{[v']}) \neq 0.$$

          Let $u=u' \cup F$ and $v=v' \cup F$. By
          Lemma~\ref{l:induced-localization} we have that 
          $$\Gamma_{[v']}=\left(\Delta_{[v]}\sm \tuple{F}\right ) _
          {\overline{F}} \mbox{ and }
          \Gamma_{[u']}=\left(\Delta_{[u]}\sm \tuple{F}\right ) _
          {\overline{F}} $$
          and $\D_{[u]}$ and $\D_{[v]}$ have $|u|$ and $|v|$ vertices,
          respectively.

                Proposition~\ref{p:local-prop} implies that
                $$\beta_{a,|u|}(\Delta_{[u]}) \neq 0 \mbox{ and }
                \beta_{b+1,|v|}(\Delta_{[v]}) \neq 0. $$

          Now we focus on $\Delta_{[v]}$, which contains the facet
          $F$. By Theorem~\ref{t:comp-facet}, $\Delta_{[F]}$ has a
          complement $\D_{[w]}$ in $\Delta_{[v]}$ with
          $\beta_{b,|w|}(\Delta_{[w]}) \neq 0.$

          We show that $\D_{[w]}$ and $\D_{[u]}$ are complements in
          $\D$.  From the fact that $\D_{[w]}$ and $\Delta_{[F]}$ are
          complements in $\Delta_{[v]}$ we see
          $$F \not \subseteq w \mbox{ and } \ w \cup F = v = v' \cup
          F.$$
          But $ v' \cap F = \emptyset$ so $w \supseteq v'$. We write
          $w=v' \cup w'$ where $w' \subset F$.

          It is clear that $w \cup u= V(\D)$. If
          $G$ is a facet of $\D$ and $$G \subset (u \cap w) = (u' \cup
          F) \cap (v' \cup w') \subset (u' \cap v') \cup F$$ then $(G \sm F) \subseteq u' \cap v'$. Now 
          there is a facet $H$ of $\D$ such that $(H \sm F) \subseteq
          (G \sm F) \mbox{ and } (H \sm F) \mbox{ is a facet of }
          \Gamma.$ But then, since $(H \sm F) \subseteq u' \cap v'$,
          we have that $(H \sm F)$ is a facet of both $\Gamma_{[u']}$
          and $\Gamma_{[v']}$, which contradicts the fact that these
          two are complements in $\Gamma$.

          So $\D_{[w]}$ and $\D_{[u]}$ have no facets in common, and
          are therefore complements in $\D$.
          
         \end{proof}

         \begin{remark} It must be noted that under the assumptions of
           Theorem~\ref{t:complement-nonzero}, not every $\m_u \in \LCM(I)$ with
             $\beta_{a,|u|}(\D_{[u]}) \neq 0$ has a complement
             $\m_v$ with $\beta_{b,|v|}(\D_{[v]}) \neq 0$.
 
             For example consider the ideal $I=(ab,bc,cd,de)$. The only
             complement of $m_u=bcd$ is $m_v=abde$. We have
             $$\beta_{3,5}(S/I) \neq 0 \mbox{ and } \beta_{2,3}(S/(bc,cd))\neq 0$$ but
             $$\beta_{1,4}(S/(ab,de))=0.$$
         \end{remark}

 \begin{theorem}[Subadditivity of Betti numbers of forests]
   Let $\D$ be a simplicial forest with facet ideal $I$ and suppose
   $i$ is at most the projective dimension of $S/I$, and $i=a+b$ where
   $a,b >0$. Then $t_a(I) +t_b(I) \geq t_{a+b}(I)$.
 \end{theorem}

 \begin{proof} Assume without loss of generality that $\D$ has $n$
   vertices and $\beta_{i,n}(\D) \neq 0$. If $\D$ has only one facet,
   there is nothing to prove as $i=1$. If $\D$ has more than one
   facet, by Theorem~\ref{t:complement-nonzero} there are complements
   $\D_{[u]}$ and $\D_{[v]}$ in $\D$ with
   $\beta_{a,|u|}(\D_{[u]}) \neq 0$ and
   $\beta_{b,|v|}(\D_{[v]}) \neq 0$. It follows that
   $$t_a(I)+t_b(I) \geq |u|+|v|\geq n=t_{a+b}(I).$$
     \end{proof}


\end{document}